\documentclass[12pt,reqno]{amsart}
\usepackage{amssymb}
\usepackage{amsmath}
\usepackage[cal=boondoxo]{mathalfa}
\usepackage{xcolor}

\usepackage{hyperref}

\input xy
\xyoption{all}

\newtheorem{theorem}{Theorem}[section]
\newtheorem{lemma}[theorem]{Lemma}
\newtheorem{proposition}[theorem]{Proposition}
\newtheorem{corollary}[theorem]{Corollary}

\theoremstyle{definition}
\newtheorem{definition}[theorem]{Definition}
\newtheorem{remark}[theorem]{Remark}

\numberwithin{equation}{section}

\begin{document}

\baselineskip=15pt

\title[Symplectic geometry of moduli of framed Higgs bundles]{Symplectic 
geometry of a moduli space of framed Higgs bundles}

\author[I. Biswas]{Indranil Biswas}

\address{School of Mathematics, Tata Institute of Fundamental
Research, Homi Bhabha Road, Mumbai 400005, India and
Mathematics Department, EISTI-University Paris-Seine, Avenue du parc, 95000,
Cergy-Pontoise, France}

\email{indranil@math.tifr.res.in}

\author[M. Logares]{Marina Logares}

\address{School of Computing Electronics and Mathematics, University of Plymouth, 
Drake Circus, PL4 8AA, Plymouth, United Kingdom} 

\email{marina.logares@plymouth.ac.uk}
 
\author[A. Pe\'on-Nieto]{Ana Pe\'on-Nieto}

\address{Universit\'e de Gen\`eve ⋄ Section de Math\'ematiques,
2-4 Rue du Li\`evre ⋄ C.P. 64 ⋄ 1211 Gen\`eve 4 ⋄ Switzerland }

\email{ana.peon-nieto@unige.ch}



\begin{abstract}
Let $X$ be a compact connected Riemann surface and $D$ an effective divisor on $X$. Let ${\mathcal 
N}_H(r,d)$ denote the moduli space of $D$--twisted stable Higgs bundles (a special class of 
Hitchin pairs) on $X$ of rank $r$ and degree $d$. It is known that ${\mathcal N}_H(r,d)$ has a 
natural holomorphic Poisson structure which is in fact symplectic if and only if $D$ is the zero 
divisor. We prove that ${\mathcal N}_H(r,d)$ admits a natural enhancement to a holomorphic 
symplectic manifold which is called here ${\mathcal M}_H(r,d)$.
This ${\mathcal M}_H(r,d)$ is constructed by trivializing, over $D$, the restriction of the vector 
bundles underlying the $D$-twisted Higgs bundles; such objects
are called here as framed Higgs bundles. We also investigate the symplectic 
structure on the moduli space ${\mathcal M}_H(r,d)$ of framed Higgs bundles as well as the Hitchin system associated to it.
\end{abstract}

\maketitle

\section{Introduction}

Since their inception in \cite{H}, Hitchin systems have been a rich source of 
examples of algebraically completely integrable systems. In \cite{H}, Hitchin proved 
that the moduli space of Higgs bundles on a compact Riemann surface is a holomorphic 
symplectic variety that admits an algebraically completely integrable structure which is
provided by what is known as the Hitchin fibration. The symplectic structure of the 
moduli space of Higgs bundles arises from the natural Liouville symplectic structure on the total space 
of the cotangent bundle of the moduli space of vector bundles; this cotangent bundle is 
in fact a Zariski open dense subset of the moduli space of Higgs bundles.

Let $X$ be a compact connected Riemann surface. Fix an effective divisor $D$ on $X$. A Hitchin pair, or more precisely
a $D$-twisted Higgs bundle, on $X$ is a pair of the form $(E,\, \theta)$, where $E$ is a holomorphic vector bundle on $X$ and
$\theta$ is a holomorphic section of the vector bundle
$\text{End}(E)\otimes K_X\otimes {\mathcal O}_X(D)$ with $K_X$ being the
holomorphic cotangent bundle of $X$. Let ${\mathcal N}_H(r,d)$ be the moduli space of stable
$D$-twisted Higgs bundles on $X$
of rank $r$ and degree $d$. It is known that ${\mathcal N}_H(r,d)$ carries a natural holomorphic Poisson structure
\cite{Bo}, \cite{Ma}. This Poisson structure 
coincides with the symplectic structure on the moduli space of stable Higgs bundles
of rank $r$ and degree $d$ (constructed in \cite{H}) when $D$ is the zero divisor. It is also known that this
Poisson structure is not symplectic when the divisor $D$ is nonzero.

Given a holomorphic vector bundle $E$ on $X$ of rank $r$, a framing of $E$ over $D$ is a 
trivialization of the vector bundle $E\vert_D$ over $D$, meaning a holomorphic isomorphism 
$\delta$ of $E\vert_D$ with the trivial vector bundle ${\mathcal O}^{\oplus r}_D$. A framed 
bundle is a holomorphic vector bundle equipped with a framing. A framed Higgs bundle is a 
triple $(E,\, \delta,\, \theta)$, where $(E,\, \delta)$ is a framed bundle and $\theta$ is a 
holomorphic section of $\text{End}(E)\otimes K_X\otimes {\mathcal O}_X(D)$ as before. Let 
	${\mathcal M}_H(r,d)$ be the moduli space of semistable framed Higgs bundles on $X$ of rank $r$ 
	and degree $d$. This is a smooth, irreducible quasi-projective variety.

	Also Simpson considered moduli spaces of framed Higgs bundles \cite{Si1,Si2}. The difference between these moduli spaces and the ones hereby considered is the twist of the  Higgs fields ($K_X$ versus $K_X(D)$). Nonetheless, the construction of the quot-scheme found in \cite[\S 4]{Si1} remains valid in our case. Combining his arguments (cf. \cite[Theorem 4.1]{Si1}) with Nitsure's \cite{Nitin} for arbitrary twistings, we may deduce that ${\mathcal M}_H(r,d)$  is an irreducible quasi-projective variety.

	 It should be mentioned that a frame is also known as a $D$-level structure as studied by Seshadri in 
\cite{Se}. Markman in \cite{Ma} studied a Poisson structure on the cotangent space to the 
moduli space of vector bundles with $D$-level structures constructed by Seshadri. The 
difference between the space considered by Markman and ours is that the stability condition we 
use is stronger.

We investigate the local structure of ${\mathcal M}_H(r,d)$. The tangent space to ${\mathcal M}_H(r,d)$
at a point $(E,\, \delta,\, \theta)\, \in\, {\mathcal M}_H(r,d)$ is given by the first hypercohomology of the
complex
$$
{\mathcal C}_{\bullet}\, :\,
{\mathcal C}_{0}\,=\, \text{End}(E)\otimes{\mathcal O}_X(-D) \, \stackrel{f_\theta}{\longrightarrow}\,
{\mathcal C}_{1}\,=\, \text{End}(E)\otimes K_X\otimes{\mathcal O}_X(D)\, ,
$$
where $f_\theta(s)\, =\, \theta\circ s - s\circ\theta$ (see Corollary \ref{cor tangent vs def space}).

It turns out that ${\mathcal M}_H(r,d)$ has a holomorphic symplectic structure,
and moreover the symplectic form is exact (Theorem \ref{thm1}). We prove that 
the forgetful map from the stable locus ${\mathcal M}^s_H(r,d)$ of ${\mathcal M}_H(r,d)$ to
${\mathcal N}_H(r,d)$,
$$
(E,\, \delta,\, \theta)\, \longmapsto\, (E,\, \theta)\, ,
$$
is compatible with the Poisson structures on ${\mathcal M}^s_H(r,d)$ and ${\mathcal N}_H(r,d)$ (see Theorem
\ref{thm poisson map}). This means that the pullback, by this forgetful map, of the Poisson bracket $\{f,\, g\}$ of two
locally defined holomorphic functions $f$ and $g$ on ${\mathcal N}_H(r,d)$ coincides with the Poisson bracket
of the pullbacks of $f$ and $g$.

Finally, we study the complete integrability properties of the Hitchin system $h$ on 
$\mathcal{M}_H(r,d)$ with respect to that of $\mathcal{N}_H(r,d)$ (denoted by 
$\widetilde{h}$) when $D$ is reduced. The generic fibers of ${h}$ are torsors over Jacobian 
varieties of spectral curves (cf. Proposition \ref{prop hitchin fibers}).The symplectic form becomes degenerate when restricted to the largest abelianizable subsystem. The fibers of the 
latter system are semiabelian varieties (Proposition \ref{prop integrable subsystem}), 
corresponding to a (local) completion, in the sense that it provides a set of Poisson commuting functions consisting of those in the Hitchin system of $\mathcal{M}_{H}(r,d)$ together with some functions transversal to $\widetilde{h}$
(Proposition \ref{prop completed int system}).

\section{Framed Higgs bundles and their deformations}

\subsection{Framed Higgs bundles and their moduli spaces}

Let $X$ be a compact connected Riemann surface. The genus
of $X$ will be denoted by $g_X$. Let $K_X$ denote the holomorphic cotangent bundle
of $X$. Fix a nonzero effective divisor
\begin{equation}\label{e1}
D\, =\, \sum_{i=1}^s n_ix_i\, ,
\end{equation}
where $\{x_i\}_{i=1}^s$ are distinct points of $X$ with $n_i\, >\, 0$ for
all $i$, and $s\, \geq\, 1$. 

We shall identify $D$ with the subscheme of $X$
with structure sheaf ${\mathcal O}_D\,=\, {\mathcal O}_X/{\mathcal O}_X(-D)$.
Similarly, for any coherent analytic sheaf $F$ on $X$, its
restriction to $D$ will be denoted by
$F_D$ and, from now on we will denote
the tensor products $F\otimes_{{\mathcal O}_X}{\mathcal O}_X(-D)$ and
$F\otimes_{{\mathcal O}_X}{\mathcal O}_X(D)$ by
$F(-D)$ and $F(D)$ respectively.

We shall identify a holomorphic vector bundle with the coherent analytic sheaf
given by its locally defined holomorphic sections. In particular, for a holomorphic vector bundle
$F$ on $X$, the above restriction $F_D$ will also mean the restriction of the
holomorphic vector bundle $F$ to the subscheme $D$. For any vector bundle $F$ on $X$, its slope
is defined to be
$$
\mu(F)\,:=\, \frac{\mathrm{deg}(F)}{\mathrm{rk}(F)}\,\in\, {\mathbb Q}\, .
$$

\begin{definition}\label{def1}
Let $E$ be a holomorphic vector bundle on $X$ of rank $r$. A \textit{frame} on $E$
is an isomorphism of ${\mathcal O}_D$ modules
$$
\delta\, :\, E_D \, \longrightarrow\,  {\mathcal O}^{\oplus r}_D\, .
$$
A vector bundle with a frame will be called a \textit{framed bundle}.
\end{definition}

\begin{definition}\label{def2}
A \textit{Higgs field} on a framed bundle $(E,\, \delta)$ is a holomorphic section of
the vector bundle ${\rm End}(E)\otimes K_X(D)$. A \textit{framed Higgs bundle}
is a triple $(E,\, \delta,\, \theta)$, where $(E,\, \delta)$ is a framed bundle and 
$\theta$ is a Higgs field on $(E,\, \delta)$.
\end{definition}

A framed bundle $(E,\, \delta)$ will be called \textit{stable} (respectively, 
\textit{semistable}) if the underlying vector bundle $E$ is \textit{stable} 
(respectively, \textit{semistable}). A framed Higgs bundle $(E,\, \delta,
\, \theta)$ will be called \textit{stable} (respectively,
\textit{semistable}) if for every subbundle $0\, \not=\, F\, \subsetneq\, E$
with $\theta(F)\, \subset\, F\otimes K_X(D)$, the
inequality
$$
\mu(F)\, <\,
\mu(E)\ \ \text{\Big(respectively, }\, 
\mu(F)\, \leq\,
\mu(E){\rm \Big)}
$$
holds.

Simpson showed that the semistable framed bundles (respectively, semistable framed 
Higgs bundles) of rank $r$ and degree $d$ have a fine moduli space denoted by ${\mathcal M}(r,d)$ 
(respectively, ${\mathcal M}_H(r,d)$), which is a smooth quasiprojective irreducible 
variety \cite{Si1}, \cite{Si2}.

We shall now see that all semistable framed Higgs bundles are simple. The frame thus provides a rigidification of the moduli problem.

\begin{lemma}\label{lem0}
Let $(E,\theta)$ and $(E',\theta')$ be semistable Higgs bundles on $X$ with
\begin{equation}
\label{s1}
\mu(E)\,=\, \mu(E')\, .
\end{equation}
Let $h\, :\, E\, \longrightarrow\, E'$ be a homomorphism such that
\begin{itemize}
\item $\theta' \circ h\,=\, (h\otimes\text{Id}_{K_X(D)})\circ\theta$ as homomorphisms from $E$ to
$E'\otimes K_X(D)$, and

\item there is a point $x_0\, \in\, X$ such that $h(x_0)\,=\, 0$.
\end{itemize}
Then $h$ vanishes identically.
\end{lemma}
\begin{proof}
Assume that $h$ is not identically zero.
Let $I$ denote the image of $h$; note that $I$ is torsion-free because it is a subsheaf of $E'$. Since
$(E,\, \theta)$ and $(E',\, \theta')$ are semistable, we have
$$
\mu(E')\, \geq\,
\mu(I)\, \geq\,
\mu(E)\, ,
$$
so from \eqref{s1} it follows that $\mu(I)\,=\,
\mu(E')$. Now, let $I'$ be the subbundle of $E'$ generated by $I$; this means that
$I'\, \subset\, E'$ is the inverse image, under the quotient map
$E'\, \longrightarrow\, E'/I$, of the torsion part of $E'/I$. Then
$$
\mu(I')\, \leq\, \mu(E')\,=\,
\mu(I)\, .
$$
On the  other hand, $\mu(I')\,\geq \, \mu(I)$, because $I$ is subsheaf of $I'$ with the quotient
$I'/I$ being a torsion sheaf. Since $\mu(I')\,=\, \mu(I)$, it follows
that $\text{deg}(I')\,=\, \text{deg}(I)$, because $\text{rank}(I')\,=\,
\text{rank}(I)$. Hence we have $I'\,=\, I$. But this is a contradiction, because
$h(x_0)\,=\, 0$ and $I\, \not=\, 0$ implying that
$x_0$ lies in the support of $I'/I$. Therefore, we conclude that $h\,=\,0$. 
\end{proof}

An isomorphism between framed Higgs bundles $(E,\, \delta, \, \theta)$
and $(E',\, \delta', \, \theta')$ is a holomorphic isomorphism
$h\,:\, E\, \longrightarrow\, E'$ of vector bundles such that
\begin{itemize}
\item $h\circ\delta\,=\, \delta'$, and

\item $\theta'\circ h\,=\, (h\otimes\text{Id}_{K_X(D)})\circ\theta$.
\end{itemize}

\begin{corollary}\label{cor0}
A semistable framed Higgs bundle $(E,\, \delta, \, \theta)$
does not admit any nontrivial automorphism.
\end{corollary}

\begin{proof}
For any automorphism $h$ of $(E,\, \delta, \, \theta)$, consider
$h-\text{Id}_E$. It vanishes on $D$.
Since the effective divisor $D$ is nonzero, from Lemma \ref{lem0} it follows that
$h-\text{Id}_E\,=\, 0$.
\end{proof}

\subsection{Infinitesimal deformations}

The following results are analogous to those of \cite{Ma}, \cite[Section 3]{Bo} and 
\cite[Section 2]{BR}. In our situation the moduli space is built using
a stability condition different from that in \cite{Ma}.

Let $X(\epsilon)\,:=\,X\times{\rm Spec}({\mathbb C}[\epsilon]/\epsilon^2)$.
An infinitesimal deformation of a framed bundle $(E,\, \delta)$ (respectively, framed Higgs bundle
$(E,\, \delta, \, \theta)$) is given by an isomorphism class of a pair
$(E_\epsilon,\, \delta_\epsilon)$
(respectively, triple $(E_\epsilon,\, \delta_\epsilon,\,\theta_\epsilon)$) such that 
\begin{itemize}
\item $E_\epsilon\,\longrightarrow\, X(\epsilon)$ is a holomorphic vector bundle,

\item $\delta_\epsilon \, :\, E_\epsilon\vert_{D\times{\rm Spec}({\mathbb C}[\epsilon]/\epsilon^2)}\,\,\longrightarrow \,{\mathcal O}^{\oplus r}_{D\times{\rm Spec}({\mathbb C}[\epsilon]/\epsilon^2)}
$ is an
isomorphism,

\item (in case of framed Higgs bundles)
$\theta_\epsilon\, \in\, H^0(X(\epsilon),\,\text{End}(E_\epsilon)\otimes
q^*_0 K_X(D))$, where $q_0$ is the natural projection of $X(\epsilon)$ to $X$,

\item $(E_\epsilon,\, \delta_\epsilon)|_{X\times\{0\}}\,\cong\, (E,\, \delta)$ (respectively,
$(E_\epsilon,\, \delta_\epsilon,\,\theta_\epsilon)|_{X\times\{0\}}\,\cong\, (E,\, \delta,\,\theta)$),
where $0\,\in\, {\rm Spec}({\mathbb C}[\epsilon]/\epsilon^2)$ denotes the closed point.
\end{itemize}

\begin{lemma}\label{lem1} 
The space of all infinitesimal deformations of a framed bundle 
$(E,\, \delta)$ is identified with $H^1(X,\, {\rm End}(E)(-D))$.
\end{lemma}

\begin{proof}
Recall that the space of all infinitesimal deformations of $E$	is $H^1(X,\, \text{End}(E))$. Now, deformations of $(E,\delta)$
are deformations of $E$ together with the framing.
Take a covering of $X$ by affine open subsets $\{U_j\}_{j=1}^d$. 
A $1$-cocycle $\{s_{j,k}\}_{j,k=1}^d$ of ${\rm End}(E)$ gives an infinitesimal
deformation of $E$ via the vector bundle $E_\epsilon$ on $X(\epsilon)$ defined by
the transition functions 
${\rm Id}+\epsilon\cdot s_{j,k}$ for the local trivializations $(q^*_0E)\vert_{U_j\times{\rm Spec}({\mathbb C}[\epsilon]/\epsilon^2)}$
over $U_j\times{\rm Spec}({\mathbb C}[\epsilon]/\epsilon^2)$. Clearly, these
preserve a frame if and only if $s_{j,k}$ are
sections of $\rm{End}(E)(-D)$. Next note that a
$1$-coboundary for ${\rm End}(E)(-D)$ gives a trivial infinitesimal 
deformation of the framed bundle $(E,\, \delta)$, as it is
a trivial deformation of $E$ which preserves the framing. Now the lemma
is straight-forward.
\end{proof}

\begin{remark}
Lemma \ref{lem1} is analogous to \cite[Proposition 6.1]{Ma}, the difference being 
that the latter considers a different moduli space, the moduli space of D-level structures, where the stability condition, known as $\overline{\delta}$-stability, is as follows.

Let $D$ an effective divisor and $\overline{\delta}=\deg(D)$, a framed bundle $(E,\delta)$ with framing over $D$ is $\overline{\delta}$-stable if and only if for all subbundle $F\subset E$ it satisfies
the inequality
$$
\mu(F)<\mu(E)+\overline{\delta}\left(\frac{1}{\mathrm{rk}(F)}-\frac{1}{\mathrm{rk}(E)}\right).
$$
\end{remark}

Given a framed Higgs bundle $(E,\, \delta, \, \theta)$ on $X$, consider the following 
complex:
\begin{equation}\label{e2}
\begin{array}{cccc}
{\mathcal C}_{\bullet}\, :\,&
{\mathcal C}_{0}\,=\, \text{End}(E)(-D) \,& \stackrel{f_\theta}{\longrightarrow}\,&
{\mathcal C}_{1}\,=\, \text{End}(E)\otimes K_X(D)\, \\
&s&\longmapsto&\theta\circ s-s\circ\theta.
\end{array}
\end{equation}

It may be noted that $f_\theta (\text{End}(E)(-D))\, \subset\, \rm{End}(E)\otimes K_X
\, \subset\, \text{End}(E)\otimes K_X(D)$.

\begin{lemma}\label{lem2}
The infinitesimal deformations of a framed Higgs bundle $(E,\, \delta, \, \theta)$
are parametrized by the first hypercohomology ${\mathbb H}^1({\mathcal C}_{\bullet})$
of the complex ${\mathcal C}_{\bullet}$ in \eqref{e2}.
\end{lemma}

\begin{proof}
Its proof is very similar to the proof of \cite[Theorem 2.3]{BR} (also given in
\cite{Bo}, \cite{Ma}).

Proceeding as in the proof of \cite[Theorem 2.3]{BR}, we check that if $\{U_j\}_{j=1}^d$ is an affine open covering of $X$ such that
\begin{itemize}
\item each point of $\{x_i\}_{i=1}^s$ lies in exactly one of these open subsets
$\{U_j\}_{j=1}^d$, and

\item $E\vert_{U_j}\,\cong\, {\mathcal O}^{\oplus r}_{U_j}$ for every $j$,
\end{itemize}
then
$\{U_j(\epsilon)\,:=\, U_j\times {\rm Spec}({\mathbb C}[\epsilon]/\epsilon^2)\}_{j=1}^d$
is an affine open covering of $X(\epsilon)$ satisfying similar conditions. From Lemma
\ref{lem1} we know that the isomorphism class of $E_\epsilon$ is
determined by the element of $H^1(X,\, {\rm End}(E)(-D))$ defined by
a $1$-cocycle $\eta_{jk}$. The Higgs field $\theta_\epsilon$ is then given by
$\theta+\gamma_j\epsilon$ on $U_j(\epsilon)$, where $$\gamma_j\,\in\, H^0(U_j,\,
\text{End}(E)\otimes K_X(D)\vert_{U_j})\, .$$ The compatibility
condition for $\eta_{jk}$ and $\gamma_j$ coincides with the condition that $(\eta_{jk},\, \gamma_j)$ defines
an element of ${\mathbb H}^1({\mathcal C}_{\bullet})$; see
\cite[Theorem 2.3]{BR}, \cite[Theorem 2.5]{Bi} for details.
\end{proof}

The dual of the complex ${\mathcal C}_{\bullet}$ in \eqref{e2} is the complex of coherent sheaves on $X$
\begin{equation}\label{dc}
\mathcal{C}^{\bullet}\, :\, \mathcal{C}^{0}\,:=\,(\text{End}(E)\otimes K_X(D))^*\otimes K_X \,
\stackrel{(f_\theta)^*\otimes\text{Id}_{K_X}}{\longrightarrow}\,
\mathcal{C}^1\,:=\, \text{End}(E)(-D)^*\otimes K_X
\end{equation}

To clarify, $\mathcal{C}^{0}$ and $\mathcal{C}^1$ are at the $0$-th and first position respectively.

Since $\text{End}(E)\,=\, \text{End}(E)^*$, we have the following:

\begin{lemma}\label{lm self dual complex}
There is an isomorphism of complexes
\begin{equation}\label{complexes}
\mathcal{C}_\bullet \,\cong\, \mathcal{C}^\bullet
\end{equation}
given by the commutative diagram
$$
\begin{matrix}
{\rm End}(E)(-D) & \stackrel{f_\theta}{\longrightarrow}& {\rm End}(E)\otimes K_X(D)\\
\Vert && ~\Vert \\
({\rm End}(E)\otimes K_X(D))^*\otimes K_X & \stackrel{(f_\theta)^*\otimes
{\rm Id}_{K_X}}{\longrightarrow}& {\rm End}(E)(-D)^*\otimes K_X.
\end{matrix}
$$
\end{lemma}

\begin{corollary}\label{cor tangent vs def space}
Let $(E,\,\delta,\,\theta)\,\in\,\mathcal{M}_H(r,d)$. Deformations
of $(E,\,\delta,\,\theta)$ are unobstructed. The tangent space of
$\mathcal{M}_H(r,d)$ at
$(E,\,\delta,\,\theta)$ is identified with $\mathbb{H}^1(\mathcal{C}_\bullet)$.
\end{corollary}

\begin{proof}
In view of Lemma \ref{lem2}, it suffices to show that there is no obstruction to the infinitesimal deformations of a framed Higgs bundle.
From Lemma \ref{lem0} it follows immediately that
\begin{equation}\label{hcv}
{\mathbb H}^0({\mathcal C}_{\bullet})\,=\, 0\, ,
\end{equation}
where ${\mathcal C}_{\bullet}$ is the complex in \eqref{e2}. Serre duality
says that ${\mathbb H}^2({\mathcal C}_{\bullet})\,=\, {\mathbb
H}^0({\mathcal C}^{\bullet})^*$. Hence using
Lemma \ref{lm self dual complex} and \eqref{hcv} it follows that
\begin{equation}\label{eq no hypercohomologies}
{\mathbb H}^2({\mathcal C}_{\bullet})\,=\, 0\, .
\end{equation}
Hence there is no obstruction to deformations.
\end{proof}

\begin{lemma}\label{cor cotangent zar dense}
Let $\mathcal{M}(r,d)$ be the moduli space of stable framed bundles of rank $r$ and degree $d$. There is a tautological embedding 
\begin{equation}\label{eq iota}
\iota\, :\, T^*{\mathcal M}(r,d) \, \hookrightarrow\, {\mathcal M}_H(r,d)
\end{equation}
whose image is a Zariski open dense subset.
\end{lemma}

\begin{proof}
From Lemma \ref{lem1} it follows that $T_{(E,\delta)}{\mathcal M}(r,d)\,=\, H^1(X,\, 
{\rm End}(E)(-D))$. Also, we have $H^1(X,\, {\rm End}(E)(-D))^*\,=\, H^0(X,\, {\rm 
End}(E)\otimes K_X(D))$ (Serre duality). Consequently, the total space of the cotangent 
bundle $T^*{\mathcal M}(r,d)$ embeds into ${\mathcal M}_H(r,d)$;
this embedding will be denoted by $\iota$. From openness of 
the semistability condition, \cite[p.~635, Theorem~2.8(B)]{Maru}, it follows 
that the image of $\iota$ is a Zariski open subset of ${\mathcal M}_H(r,d)$. This also
implies that the image of $\iota$ is Zariski dense in ${\mathcal M}_H(r,d)$, because
${\mathcal M}_H(r,d)$ is irreducible.
\end{proof}

The dimension of the moduli space $\mathcal{M}_{H}(r,d)$ can now be calculated.

\begin{proposition}\label{dim}
Let $n\, =\, \sum_{i=1}^s n_i$. The dimension of $\mathcal{M}_{H}(r,d)$ is $2r^{2}(g_X+n-1)$,
where $g_X$ as before is the genus of $X$.
\end{proposition}

\begin{proof}
By Corollary \ref{cor tangent vs def space}
$$
\dim \mathcal{M}_{H}(r,d)\,=\,\dim\mathbb{H}^1(\mathcal{C}_\bullet)\, .
$$
Now, by \eqref{hcv} and \eqref{eq no hypercohomologies}
$${\mathbb H}^0({\mathcal C}_{\bullet})\,=\, 0\,=\, {\mathbb H}^2({\mathcal C}_{\bullet})\, .$$
Hence we have
$$\dim {\mathbb H}^1({\mathcal C}_{\bullet})\, =\, -\chi({\mathcal C}_{\bullet})\,=\, 
-\chi(\text{End}(E)(-D)) +\chi(\text{End}(E)\otimes K_X(D))\, .
$$
By Riemann--Roch, $\chi(\text{End}(E)\otimes K_X(D))\,=\, -\chi (\text{End}(E)(-D))\,=\, r^{2}(g_X+n-1)$.
Therefore, it follows that
$\dim {\mathbb H}^1({\mathcal C}_{\bullet})\, =\, 2r^{2}(g_X+n-1)$.
\end{proof}

\section{Symplectic geometry}

\subsection{A symplectic form on $\mathcal{M}_H(r,d)$.}\label{sect symplectic form}

In this section we construct a symplectic form on $\mathcal{M}_H(r,d)$ following analogous arguments to those in \cite{Bi, Bo, Ma}. In view of 
Corollary \ref{cor tangent vs def space}, we shall start by constructing a nondegenerate 
antisymmetric bilinear form on $\mathbb{H}^1(\mathcal{C}_\bullet)$.

The isomorphism of complexes ${\mathcal C}_{\bullet}\, \stackrel{\sim}{\longrightarrow}\,
\mathcal{C}^{\bullet}$ in \eqref{complexes} produces an isomorphism of hypercohomologies
\begin{equation}\label{mapB}
B_{\theta}\,:\, \mathbb{H}^{1}(\mathcal{C}_{\bullet})\,\stackrel{\sim}{\longrightarrow}\,
\mathbb{H}^{1}(\mathcal{C}^{\bullet})\, .
\end{equation}
On the other hand, Serre duality gives an isomorphism
\begin{equation}\label{sdis}
\mathbb{H}^{1}(\mathcal{C}^{\bullet})
\, \stackrel{\sim}{\longrightarrow}\,\mathbb{H}^{1}(\mathcal{C}_{\bullet})^*\, .
\end{equation}
We recall the explicit description of Serre duality in this case. First note that in view of the isomorphism
in \eqref{mapB}, the isomorphism in \eqref{sdis} is uniquely determined by the corresponding isomorphism
\begin{equation}\label{mapBl}
\mathbb{H}^{1}(\mathcal{C}_{\bullet})\,\stackrel{\sim}{\longrightarrow}\, \mathbb{H}^{1}(\mathcal{C}_{\bullet})^*
\end{equation}
constructed using \eqref{mapB}. To construct the isomorphism in \eqref{mapBl}, consider the tensor product of complexes
$\mathcal{C}_{\bullet}\otimes \mathcal{C}_{\bullet}$:
$$
(\mathcal{C}_{\bullet}\otimes \mathcal{C}_{\bullet})_0\,=\,
\text{End}(E)(-D)\otimes \text{End}(E)(-D)\, \stackrel{(f_\theta\otimes {\rm Id})+
({\rm Id}\otimes f_\theta)}{\longrightarrow}\, (\mathcal{C}_{\bullet}\otimes \mathcal{C}_{\bullet})_1
$$
$$
=\, (\text{End}(E)\otimes K_X(D)\otimes \text{End}(E)(-D))\oplus (\text{End}(E)(-D)\otimes \text{End}(E)\otimes K_X(D))
$$
$$
\stackrel{({\rm Id}\otimes f_\theta)-
(f_\theta\otimes {\rm Id})}{\longrightarrow}\, (\mathcal{C}_{\bullet}\otimes \mathcal{C}_{\bullet})_2\,=\,
(\text{End}(E)\otimes K_X(D))\otimes (\text{End}(E)\otimes K_X(D))\, .
$$
We also have the homomorphism
$$
\rho\, :\, (\text{End}(E)\otimes K_X(D)\otimes \text{End}(E)(-D))\oplus (\text{End}(E)(-D)\otimes \text{End}(E)\otimes K_X(D))
$$
$$
\longrightarrow\, K_X\, , \ \ (a_1\otimes b_1)\oplus (a_2\otimes b_2)\, \longmapsto\, \text{Tr}(a_1\circ b_1+a_2\circ b_2)\, .
$$
In the above, $a_1$ and $b_2$ are sections of $\mathrm{End}(E)\otimes K_X(D)$ and $a_2$ and $b_1$ are sections of $\mathrm{End}(E)(-D)$. 
These give the homomorphism of complexes
\begin{equation}\label{hh1}
\begin{matrix}
(\mathcal{C}_{\bullet}\otimes \mathcal{C}_{\bullet})_0 & \longrightarrow &
(\mathcal{C}_{\bullet}\otimes \mathcal{C}_{\bullet})_1 & \longrightarrow & 
(\mathcal{C}_{\bullet}\otimes \mathcal{C}_{\bullet})_2\\
\Big\downarrow && ~\Big\downarrow\rho && \Big\downarrow\\
0 & \longrightarrow & K_X & \longrightarrow & 0
\end{matrix}
\end{equation}
Now we have the composition of homomorphisms of hypercohomologies
\begin{equation}\label{ch}
{\mathbb H}^1({\mathcal C}_{\bullet})\otimes
{\mathbb H}^1({\mathcal C}_{\bullet})\, \longrightarrow\, {\mathbb H}^2(\mathcal{C}_{\bullet}\otimes \mathcal{C}_{\bullet})
\, \stackrel{\rho'}{\longrightarrow}\,{\mathbb H}^2(0\rightarrow K_X \rightarrow 0)\,=\,
H^1(X, \, K_X) \,=\, {\mathbb C}\, ,
\end{equation}
where $\rho'$ is given by the homomorphism of complexes in \eqref{hh1}.
Let
\begin{equation}\label{f4}
\Psi_\theta\, :\, {\mathbb H}^1({\mathcal C}_{\bullet})\otimes
{\mathbb H}^1({\mathcal C}_{\bullet})\, \longrightarrow\, \mathbb C
\end{equation}
be the bilinear form constructed in \eqref{ch}.

The earlier mentioned Serre duality in \eqref{mapBl}
is given by the bilinear form $\Psi_\theta$ in \eqref{f4}. Note that $\Psi_\theta$ is nondegenerate because the
homomorphism in \eqref{mapBl} is an isomorphism. Also, the anti-symmetry of $\Psi_\theta$ comes from its construction and it is related to the graded symmetry of the cup-product in cohomology.

Therefore, we have the following:

\begin{proposition}\label{prop 2 form}
There is a nondegenerate holomorphic two form $\Psi$ on $\mathcal{M}_H(r,d)$ whose
evaluation at any point $(E,\,\delta,\,\theta)\, \in\, \mathcal{M}_H(r,d)$ is $\Psi_\theta$
in \eqref{f4}.
\end{proposition}

\begin{proof}
The above pointwise construction of $\Psi_\theta$ clearly works for families of framed
Higgs bundles. Recall from Corollary \ref{cor tangent vs def space} that $\mathbb{H}^1(\mathcal{C}_\bullet)$ is the tangent space to
$\mathcal{M}_H(r,d)$ at the point $(E,\,\delta,\,\theta)$.
\end{proof}

Before exploring $\Psi_\theta$ further, we introduce a one form on $\mathcal{M}_H(r,d)$.

Consider the short exact sequence of complexes of coherent sheaves on $X$
$$
\begin{matrix}
&0 && 0\\
&\Big\downarrow && \Big\downarrow\\
{\mathcal D}_\bullet:&0 & \longrightarrow & \text{End}(E)\otimes K_X(D)\\
&\Big\downarrow && \Big\downarrow\\
{\mathcal C}_\bullet:&\text{End}(E)(-D) & \stackrel{f_\theta}{\longrightarrow}& \text{End}(E)\otimes K_X(D)\\
&\Big\downarrow && \Big\downarrow\\
{\mathcal E}_\bullet:&\text{End}(E)(-D) &\longrightarrow & 0\\
&\Big\downarrow && \Big\downarrow\\
&0 && 0
\end{matrix}
$$
In the above diagram, both of the complexes ${\mathcal D}_\bullet$ and ${\mathcal E}_\bullet$ have only one nonzero term.
Hence their hypercohomologies are just (shifted) cohomologies of the
single nonzero term. Thus we have an associated long exact sequence of hypercohomologies
$$
0\, \longrightarrow\, {\mathbb H}^0({\mathcal C}_{\bullet})\, \longrightarrow\, {\mathbb H}^0({\mathcal E}_\bullet)\,=\,
H^0(X,\, \text{End}(E)(-D)) \, \longrightarrow\, {\mathbb H}^1({\mathcal D}_\bullet)
$$
\begin{equation}\label{e3}
=\, H^0(X,\, \text{End}(E)\otimes K_X(D))\, \stackrel{a}{\longrightarrow}\,{\mathbb H}^1({\mathcal C}_{\bullet})
\,\stackrel{b}{\longrightarrow}\, {\mathbb H}^1({\mathcal E}_\bullet)\,=\, H^1(X,\, \text{End}(E)(-D))
\end{equation}
$$
 \longrightarrow\,{\mathbb H}^2({\mathcal D}_\bullet)\,=\, H^1(X,\, \text{End}(E)\otimes K_X(D))
\, \longrightarrow\,{\mathbb H}^2({\mathcal C}_{\bullet})\, \longrightarrow\, 0\, .
$$
The homomorphism $b$ in \eqref{e3} is the forgetful map that sends an
infinitesimal deformation of $(E,\, \delta, \, \theta)$ to the infinitesimal
deformation of $(E,\, \delta)$ obtained from it by simply forgetting the Higgs field.
The homomorphism $a$ in \eqref{e3} is the one that sends a section $$u\, \in\,
H^0(X,\, \text{End}(E)\otimes K_X(D))$$ to the infinitesimal deformation of
$(E,\, \delta, \, \theta)$ defined by $t\, \longmapsto\, (E,\, \delta, \, \theta+tu)$.

Since $H^0(X,\, {\rm End}(E)\otimes K_X(D))\,=\, H^1(X,\, {\rm End}(E)(-D))^*$ (Serre duality),
we have a homomorphism
\begin{equation}\label{e4}
\Phi_\theta\, :\, {\mathbb H}^1({\mathcal C}_{\bullet})\,\longrightarrow\,{\mathbb C}\, ,
\ \ v \,\longmapsto\, \theta(b(v))\, ,
\end{equation}
where $b$ is the homomorphism in \eqref{e3}.

The above construction produces the following:

\begin{proposition}\label{prop 1 form}
There is a holomorphic one-form $\Phi$ on $\mathcal{M}_H(r,d)$ whose 
evaluation at any point $(E,\,\delta,\,\theta)\, \in\, \mathcal{M}_H(r,d)$ is $\Phi_\theta$
in \eqref{e4}.
\end{proposition}

\begin{proof}
The above pointwise construction of $\Phi_\theta$ clearly works for families of framed
Higgs bundles.
\end{proof}

Using the arguments in \cite[Theorem 4.6]{BR} it can be shown that $\Psi_\theta$ defined in \eqref{f4} is the evaluation at $(E,\delta,\theta)$ of the form $d\Phi$. We hereby give a somewhat simplified proof (see Theorem \ref{thm1}).

\subsection{Relating to the Liouville symplectic form}

For any complex manifold $N$, the total space $T^*N$ of the holomorphic cotangent
bundle of $N$ is equipped with the Liouville holomorphic $1$-form, which will be denoted by $\eta_N$. The
holomorphic 2-form $d\eta_N$ is the Liouville symplectic form on $T^*N$.

Recall from Lemma \ref{cor cotangent zar dense} that we have a Zariski dense open subset
$$\iota\,:\,T^*\mathcal{M}(r,d)
\,\subset\, \mathcal{M}_H(r,d)\, .$$
In this subsection we will prove that $\iota^*\Phi$ is the Liouville $1$-form and
$\iota^*\Psi$ is the Liouville symplectic form on the cotangent bundle $T^*\mathcal{M}(r,d)$.

\begin{proposition}\label{prop1}
The Liouville $1$-form on $T^*{\mathcal M}(r,d)$ coincides with the pullback
$\iota^*\Phi$, where $\iota$ is the embedding in \eqref{eq iota} and $\Phi$ is the 1-form in Proposition \ref{prop 1 form}.

The pullback $\iota^*\Psi$ is the Liouville symplectic form on $T^*{\mathcal M}(r,d)$,
meaning $\iota^*\Psi\,=\, d\iota^*\Phi$, where $\Psi$ is the 2-form in Proposition \ref{prop 2 form}.
\end{proposition}

\begin{proof}
Let
$$
p\, :\, T^*{\mathcal M}(r,d)\, \longrightarrow\, {\mathcal M}(r,d)
$$
be the natural projection. For any $z\, :=\, (E,\, \delta, \, \theta)\,
\in\, T^*{\mathcal M}(r,d)\, \subset\, {\mathcal M}_H(r,d)$, let
$$
dp(z) \, :\, T_zT^*{\mathcal M}(r,d)\, \longrightarrow\, T_{p(z)} {\mathcal M}(r,d)
$$
be the differential of $p$ at $z$. We noted earlier that the
homomorphism $b$ in \eqref{e3} is the
forgetful map that sends an
infinitesimal deformation of $(E,\, \delta, \, \theta)$ to the infinitesimal
deformation of $(E,\, \delta)$ obtained from it by simply forgetting the Higgs field. This
means that $b$ coincides with the above homomorphism $dp(z)$. Now from the
definition of the Liouville $1$-form on $T^*{\mathcal M}(r,d)$, and the
construction of $\Phi_\theta$ in \eqref{e4}, it follows immediately
that $\iota^*\Phi$ is the Liouville $1$-form on $T^*{\mathcal M}(r,d)$.

The fact that $\iota^*\Psi\,=\, d\iota^*\Phi$ follows from a standard argument (see \cite[Theorem 4.3]{BR}, \cite[Proposition 7.3]{Bi},
\cite[Theorem 4.5.1]{Bo}, \cite[Corollary 7.15]{Ma}). We omit the details.
\end{proof}

\begin{theorem}\label{thm1}
The holomorphic $2$-form $\Psi$ on ${\mathcal M}_H(r,d)$ is symplectic. Moreover,
$$
\Psi\,=\,d\Phi
$$
on $\mathcal{M}_H(r,d)$.
\end{theorem}

\begin{proof}
We know that $\Psi$ is nondegenerate and antisymmetric. So it suffices to prove that
$\Psi\,=\, d\Phi$. By Proposition \ref{prop1}, the form $\Psi$ coincides
with $d\Phi$ on the dense open subset $T^*{\mathcal M}(r,d)$ of ${\mathcal M}_H(r,d)$.
This implies that $\Psi\,=\, d\Phi$ on the entire ${\mathcal M}_H(r,d)$.
\end{proof}
 
\section{Poisson maps}

In this section we compare the Poisson structure on $\mathcal{M}_H(r,d)$ with the
Poisson structure on the moduli space $\mathcal{N}_{H}(r,d)$ of stable Hitchin pairs.

For any holomorphic symplectic manifold $(M,\, \omega)$, we may construct a Poisson bracket of 
any two locally defined holomorphic functions $f,\, g\,\in\, \mathcal{O}_{M}$, by means of the
Hamiltonian vector fields $X_{f}$ and $X_{g}$ associated to them, that is,
$$
\{f,\, g\}\,:=\, X_{f} g\, =\, -X_{g} f\,=\, -\{g,\, f\}\, ;
$$
the Hamiltonian vector field $X_h$ for a function $h$
is defined by the equation $dh(v)\,=\, \omega(X_{h},\,v)$, where $v$ is any vector field
on $M$. The above pairing $\{\cdot,\cdot\}$ provides $\mathcal{O}_{M}$ with a Lie algebra structure
which satisfies the Leibniz rule which says that $\{f,\,gh\}\,=\,\{f,\,g\}h+g\{f,\, h\}$; therefore
this pairing produces a Poisson structure on $M$. 

Let $(Y_1,\, \omega_1)$ and $(Y_2,\, \omega_2)$ be two holomorphic Poisson manifolds, where
$$
\omega_1\, :\, T^*Y_1\, \longrightarrow\, TY_1\ \ \text{ and }\ \
\omega_2\, :\, T^*Y_2\, \longrightarrow\, TY_2
$$
are the holomorphic homomorphisms giving the Poisson structures. A holomorphic map
$$
\beta\, :\, Y_1\, \longrightarrow\, Y_2
$$
is said to be compatible with the Poisson structures if
$$\{f,\, g\}_2\circ\beta \,=\, \{ f\circ\beta,\, g\circ\beta\}_1$$
for all locally defined holomorphic functions $f,\, g$ on $Y_2$, where $\{-,\, -\}_1$ and
$\{-,\, -\}_2$ are the Poisson brackets on $Y_1$ and $Y_2$ respectively.

Let
\begin{equation}\label{db}
d\beta\, :\, TY_1\, \longrightarrow\, TY_2
\end{equation}
be the differential of the map $\beta$. It is straight-forward to check that the map $\beta$ is
compatible with the Poisson structures if and only if
\begin{equation}\label{pch}
d\beta(\omega_1(x)((d\beta)^*(x)(u)))\,=\, \omega_2(u)
\end{equation}
for all $u\, \in\, T^*_yY_2$ and $x\,\in\,\beta^{-1}(y)$, where $(d\beta)^*(x)\, :\, T^*_yY_2\, \longrightarrow\,
T^*_xY_1$ is the dual of the differential $d\beta (x)$ in \eqref{db}. Note that both sides of \eqref{pch} are
elements of $T_y Y_2$.

Let $\mathcal{N}_{H}(r,d)$ be the moduli space of stable Hitchin pairs $(E,\, \theta)$, where
\begin{itemize}
\item $E$ is a holomorphic vector bundle on $X$ of rank $r$ and degree $d$, and
	
\item $\theta\, \in\, H^0(X,\, \text{End}(E)\otimes K_X(D))$.
\end{itemize}
The choice of a section $s$ of $\mathcal{O}_X(D)$
determines a holomorphic Poisson structure on $\mathcal{N}_{H}(r,d)$ \cite[Theorem 4.5.1]{Bo}
(see also \cite{H}, \cite{La}); we will take $s$ to be the section given
by the constant function $1$. The moduli space $\mathcal{N}_{H}(r,d)$ is also investigated in \cite{BGL}. The construction of the Poisson structure
is recalled in the proof of Theorem \ref{thm poisson map}.

Let
\begin{equation}\label{w1}
\widetilde{\Psi}'\, :\, T^*\mathcal{N}_{H}(r,d)\, \longrightarrow\, T\mathcal{N}_{H}(r,d)
\end{equation}
be this Poisson structure on $\mathcal{N}_{H}(r,d)$. Let
\begin{equation}\label{w2}
\widetilde{\Psi}\, :\, T^*\mathcal{M}_{H}(r,d)\, \stackrel{\sim}{\longrightarrow}\, T\mathcal{M}_{H}(r,d)
\end{equation}
be the Poisson structure on $\mathcal{M}_{H}(r,d)$ given by the symplectic form $\Psi$ in Theorem \ref{thm1}. 

Let $\mathcal{M}^s_{H}(r,d)\, \subset\, \mathcal{M}_{H}(r,d)$ be the stable locus; it is a Zariski
open dense subset. Let
\begin{equation}\label{eq Q}
Q\, :\, \mathcal{M}^s_{H}(r,d)\, \longrightarrow\, \mathcal{N}_{H}(r,d)\, , \ \ (E,\, \delta, \, \theta)\,\longmapsto\,
(E,\, \theta)
\end{equation}
be the forgetful map that simply forgets the framing.

\begin{theorem}\label{thm poisson map}
The map $Q$ in \eqref{eq Q} is compatible with the Poisson structures $\widetilde{\Psi}'$ and
$\widetilde{\Psi}$ on $\mathcal{N}_{H}(r,d)$ and $\mathcal{M}^s_{H}(r,d)$ respectively.
\end{theorem}

\begin{proof}
We will check the criterion in \eqref{pch}.
	
	Consider the complex of coherent sheaves on $X$
\begin{equation}\label{2cli}
	\mathcal{C}^{\bullet}_\tau\, :\, \mathcal{C}^{0}_\tau\,:=\,(\text{End}(E)\otimes K_X(D))^*\otimes K_X \,
	\stackrel{(f_\theta)^*\otimes\text{Id}_{K_X}}{\longrightarrow}\,
	\mathcal{C}^1_\tau\,:=\, \text{End}(E)^*\otimes K_X\, .
\end{equation}
The identity map of $(\text{End}(E)\otimes K_X(D))^*\otimes K_X$
and the natural inclusion of $\text{End}(E)^*\otimes K_X$ in
$\text{End}(E)^*\otimes K_X(D)$ (recall that the divisor $D$ is effective) together
produce a homomorphism of complexes
$$
	\xi\, :\, \mathcal{C}^{\bullet}_\tau\, \longrightarrow\, \mathcal{C}^{\bullet}
	$$
(the complex $\mathcal{C}^{\bullet}$ is constructed in \eqref{dc}); in other words, the commutative diagram
	$$
	\begin{matrix}
	(\text{End}(E)\otimes K_X(D))^*\otimes K_X & \stackrel{(f_\theta)^*\otimes\text{Id}_{K_X}}{\longrightarrow}&
	\text{End}(E)^*\otimes K_X\\
	\Vert && \Big\downarrow \\
	(\text{End}(E)\otimes K_X(D))^*\otimes K_X & \stackrel{(f_\theta)^*\otimes
		\text{Id}_{K_X}}{\longrightarrow}& \text{End}(E)(-D)^*\otimes K_X= \text{End}(E)^*\otimes K_X(D)
	\end{matrix}
	$$
defines $\xi$. Let
	\begin{equation}\label{xie}
	\xi_*\, :\, {\mathbb H}^1(\mathcal{C}^{\bullet}_\tau)\, \longrightarrow\,{\mathbb H}^1(\mathcal{C}^{\bullet})
	\end{equation}
	be the homomorphism of hypercohomologies induced by the above homomorphism $\xi$ of complexes.
	
We note that $T^*_{Q(E, \delta, \theta)}\mathcal{N}_{H}(r,d)\,=\, T^*_{(E,\theta)}\mathcal{N}_{H}(r,d)
\,=\, {\mathbb H}^1(\mathcal{C}^{\bullet}_\tau)$ \cite[Proposition 3.1.10]{Bo},
\cite[Corollary 7.5]{Ma}, \cite[Theorem 2.3]{BR}. On the other hand,
$$
{\mathbb H}^1(\mathcal{C}^{\bullet})\,=\, {\mathbb H}^1(\mathcal{C}_{\bullet})^*\,=\,
T^*_{(E, \delta, \theta)}\mathcal{M}^s_{H}(r,d)
$$
(Corollary \ref{cor tangent vs def space} and \eqref{sdis}). Take any
$$
w\, \in\, {\mathbb H}^1(\mathcal{C}^{\bullet}_\tau)\,=\,
T^*_{Q(E, \delta, \theta)}\mathcal{N}_{H}(r,d)\, .
$$

We will show that 
\begin{equation}\label{sh2}
(dQ)^*(E, \delta, \theta)(w)\,=\, \xi_*(w)\, ,
\end{equation}
where $\xi_*$ is constructed in \eqref{xie}.

To prove \eqref{sh2}, consider the complex of coherent sheaves on $X$
\begin{equation}\label{3cli}
	{\mathcal C}^\tau_{\bullet}\, :\,
	{\mathcal C}^\tau_{0}\,=\, \text{End}(E) \, \stackrel{f_\theta}{\longrightarrow}\,
	{\mathcal C}^\tau_{1}\,=\, \text{End}(E)\otimes K_X(D)\, .
\end{equation}
	We have
	$$
	T_{Q(E, \delta, \theta)}\mathcal{N}_{H}(r,d)\,=\, T_{(E, \theta)}\mathcal{N}_{H}(r,d)\,=\,
{\mathbb H}^1({\mathcal C}^\tau_{\bullet})
	$$
\cite{Bo}, \cite{Ma}, \cite{BR}. Next we note that the identity map of
$\text{End}(E)\otimes K_X(D)$ and the natural inclusion
of $\text{End}(E)(-D)$ in $\text{End}(E)$ together produce a
homomorphism of complexes
\begin{equation}\label{zet0}
\zeta\, :\, \mathcal{C}_{\bullet}\, \longrightarrow\, \mathcal{C}^\tau_{\bullet}
\end{equation}
($\mathcal{C}_{\bullet}$ is constructed in \eqref{e2}); in other words, we have the commutative diagram
$$
\begin{matrix}
	\text{End}(E)(-D) & \stackrel{f_\theta}{\longrightarrow}&
	\text{End}(E)\otimes K_X(D)\\
	\Big\downarrow && \Vert \\
	\text{End}(E) & \stackrel{f_\theta}{\longrightarrow}& \text{End}(E)\otimes K_X(D)
	\end{matrix}
	$$
that defines $\zeta$. Let
\begin{equation}\label{zet}
\zeta_*\, :\, {\mathbb H}^1(\mathcal{C}_{\bullet})\, \longrightarrow\,{\mathbb H}^1(\mathcal{C}^\tau_{\bullet})
\end{equation}
be the homomorphism of hypercohomologies induced by the above homomorphism $\zeta$
of complexes. It can be shown that in terms of the identifications
$$
T_{Q(E, \delta, \theta)}\mathcal{N}_{H}(r,d)\,=\, {\mathbb H}^1({\mathcal C}^\tau_{\bullet})\ \
\text{ and } \ \, T_{(E, \delta, \theta)}\mathcal{M}^s_{H}(r,d)\,=\, {\mathbb H}^1({\mathcal C}_{\bullet})\, ,$$
the differential $$dQ(E, \delta, \theta)\, :\, T_{(E, \delta, \theta)}\mathcal{M}^s_{H}(r,d)\, \longrightarrow\,
T_{Q(E, \delta, \theta)}\mathcal{N}_{H}(r,d)$$
of $dQ$ at $(E, \,\delta,\, \theta)\, \in\, \mathcal{M}^s_{H}(r,d)$ coincides with the homomorphism
$\zeta_*$ constructed in \eqref{zet}. Indeed, this follows immediately from the constructions of the map $Q$
and the homomorphism $\zeta_*$. Finally, \eqref{sh2} follows from the isomorphism in \eqref{sdis} and the definition of
the homomorphism $(dQ)^*(E, \delta, \theta)$ as the dual homomorphism.
	
At this point, we shall recall the construction of the Poisson structure $\widetilde{\Psi}'$ in
\eqref{w1}. For this, consider the complexes ${\mathcal C}^{\bullet}_\tau$ and ${\mathcal C}^\tau_{\bullet}$
constructed in \eqref{2cli} and \eqref{3cli} respectively. We have a homomorphism of complexes
\begin{equation}\label{4c}
\varpi\, :\, {\mathcal C}^{\bullet}_\tau\, \longrightarrow \,{\mathcal C}^\tau_{\bullet}
\end{equation}
defined by the following diagram of homomorphisms:
$$
\begin{matrix}
(\text{End}(E)\otimes K_X(D))^*\otimes K_X = \text{End}(E)(-D) & \stackrel{(f_\theta)^*\otimes\text{Id}_{K_X}}{\longrightarrow}&
\text{End}(E)^*\otimes K_X\\
\Big\downarrow && \Big\downarrow \\
\text{End}(E) & \stackrel{f_\theta}{\longrightarrow}& \text{End}(E)\otimes K_X(D)
\end{matrix}
$$
note that here it is used that
\begin{itemize}
\item $\text{End}(E)^*\,=\, \text{End}(E)$, and

\item the sheaves ${\mathcal O}_X(-D)$ and ${\mathcal O}_X$ are
contained in ${\mathcal O}_X$ and ${\mathcal O}_X(D)$ respectively.
\end{itemize}
Let
\begin{equation}\label{vps}
\varpi_*\, :\, {\mathbb H}^1({\mathcal C}^{\bullet}_\tau)\, \longrightarrow\,
{\mathbb H}^1({\mathcal C}^\tau_{\bullet})
\end{equation}
be the homomorphism of hypercohomologies induced by the
homomorphism of complexes $\varpi$ in \eqref{4c}. Since
$$
{\mathbb H}^1(\mathcal{C}^{\bullet}_\tau)\,=\, T^*_{(E,\theta)}\mathcal{N}_{H}(r,d)\ \ \text{ and }\ \
{\mathbb H}^1({\mathcal C}^\tau_{\bullet})\,=\, T_{(E,\theta)}\mathcal{N}_{H}(r,d)\, ,
$$
the above homomorphism $\varpi_*$ produces a homomorphism
$$
T^*_{(E,\theta)}\mathcal{N}_{H}(r,d)\, \longrightarrow\,T_{(E,\theta)}\mathcal{N}_{H}(r,d)\, .
$$
This homomorphism coincides with $\widetilde{\Psi}'(E,\theta)$ in \eqref{w1}.

Finally, consider the isomorphism $\widetilde{\Psi}$ in \eqref{w2}. We note
that $\widetilde{\Psi}$ coincides with the homomorphism of hypercohomologies
induced by the isomorphism of complexes in Lemma \ref{lm self dual complex}.
Using this and \eqref{mapBl} it follows that the composition
$$
\widetilde{\Psi}\circ (dQ)^*(E, \delta, \theta)\, :\, 
T^*_{Q(E, \delta, \theta)}\mathcal{N}_{H}(r,d) \,=\, {\mathbb H}^1(\mathcal{C}^{\bullet}_\tau)\,\longrightarrow\,
{\mathbb H}^1({\mathcal C}_{\bullet})\,=\,T_{(E, \delta, \theta)}\mathcal{M}^s_{H}(r,d)
$$
coincides with the homomorphism of hypercohomologies associated to the following
natural homomorphism of complexes:
$$
\begin{matrix}
(\text{End}(E)\otimes K_X(D))^*\otimes K_X & \stackrel{(f_\theta)^*\otimes\text{Id}_{K_X}}{\longrightarrow}&
\text{End}(E)^*\otimes K_X\\
\Vert && \Big\downarrow \\
\text{End}(E)(-D) & \stackrel{f_\theta}{\longrightarrow}& \text{End}(E)\otimes K_X(D)
\end{matrix}
$$
(as before, we use that $\text{End}(E)^*\,=\, \text{End}(E)$ and that ${\mathcal O}_X$ is
contained in ${\mathcal O}_X(D)$). From the above description of
$\widetilde{\Psi}\circ (dQ)^*(E, \delta, \theta)$, and the earlier observation that
$dQ(E, \delta, \theta)$ coincides with the homomorphism $\zeta_*$
in \eqref{zet}, it follows that the composition
\begin{equation}\label{w4}
(dQ)\circ\widetilde{\Psi}\circ (dQ)^*(E, \delta, \theta)\, :\, 
T^*_{Q(E, \delta, \theta)}\mathcal{N}_{H}(r,d) \,=\, {\mathbb H}^1(\mathcal{C}^{\bullet}_\tau)\,\longrightarrow\,
{\mathbb H}^1({\mathcal C}^\tau_{\bullet})\,=\,T_{Q(E, \delta, \theta)}\mathcal{N}_{H}(r,d)
\end{equation}
coincides with the homomorphism $\varpi_*$ in \eqref{vps} of hypercohomologies associated to the homomorphism $\varpi$ of
complexes in \eqref{4c}. Consequently, the homomorphism in \eqref{w4} coincides with
$\widetilde{\Psi}'$ in \eqref{w1}. This completes the proof.
\end{proof}
We end this section by studying the structure of the map $Q$.

\begin{proposition}\label{prop Q is torsor}
Assume that the divisor $D$ is reduced, thus $n_i\,=\, 1$ for all
$1\,\leq\, i\,\leq\, s$ and $n=s$. Consider the diagonal embedding of the center
$\mathbb{C}^\times$ of $\mathrm{GL}(r,\mathbb{C})$ in $\mathrm{GL}(r,\mathbb{C})^n$ 
The morphism $Q$ turns $\mathcal{M}^s_{H}(r,d)$ into a $\mathrm{GL}(r,\mathbb{C})^n/
\mathbb{C}^\times$-torsor over $\mathcal{N}_{H}^s(r,d)$.
\end{proposition}

\begin{proof}
Take any $(E,\, \theta)\,\in\,\mathcal{N}_H(r,d)$.
The group $$\text{Aut}({\mathcal O}^{\oplus r}_D)\,\stackrel{\sim}{\longrightarrow}\,
\prod_{i=1}^n \text{Aut}({\mathcal O}^{\oplus r}_{x_i})\,=\, \text{GL}(r,
{\mathbb C})^n$$ (the above isomorphism holds because $D$ is reduced)
acts on the space of framings on $E$ (note that $n\,=\, s$); see Definition
\ref{def1}. This action factors through the quotient group $\mathrm{GL}(r,\mathbb{C})^n/
\mathbb{C}^\times$ because scalars act as automorphisms of $(E,\, \theta)$.
So $\mathrm{GL}(r,\mathbb{C})^n/\mathbb{C}^\times$ acts on the fiber of $Q$ over
the point $(E,\, \theta)\,\in\,\mathcal{N}_H(r,d)$; this action is evidently transitive.
Since $(E,\, \theta)$ is stable, it is simple, hence the above action of
$\mathrm{GL}(r,\mathbb{C})^n/\mathbb{C}^\times$ on the fiber of $Q$ over
$(E,\, \theta)$ is free.

In the proof of Theorem \ref{thm poisson map} it was observed that the
differential $dQ(E,\delta,\theta)$ coincides with the homomorphism $\zeta_*$
in \eqref{zet}. Note that $\zeta$ is injective, and the quotient of the homomorphism
$\zeta$ in \eqref{zet0}
is a torsion sheaf at the $0$-th position. As the first cohomology of a torsion sheaf
on $X$ vanishes, the first hypercohomology of the quotient complex vanishes.
Therefore, from the corresponding long exact
sequence of hypercohomologies associated to $\zeta$
it follows that the homomorphism $\zeta_*$ is surjective.
Consequently, the map $Q$ is a submersion. This completes the proof.
\end{proof}

\section{The Hitchin integrable system}

Recall from the previous section that, for any holomorphic symplectic manifold $(M,\, \omega)$, we denote by $\{\cdot,\cdot\}$ the associated Poisson bracket on $\mathcal{O}_{M}$. 

Two functions $f,\,g\in \mathcal{O}_{M}$ are said to Poisson commute if
$$
\{f,\, g\}\,=\,0\, .
$$
An algebraically completely integrable system on $M$ consists
of functions $f_1,\, \ldots ,\,f_d\,\in\, \mathcal{O}_{M}$ with $d\,=\, \frac{1}{2}\dim M$,
such that
\begin{itemize}
\item $\{f_i,\, f_j\}\,=\,0$ for all $1\, \leq\, i,\, j\, \leq\, d$,

\item the corresponding Hamiltonian vector fields $X_{f_1},\, \cdots, \, X_{f_d}$ are linearly independent
at the general point, and

\item the general fiber of the map $(f_1,\cdots,\, f_d) \,:\, M\,\longrightarrow\,
\mathbb{C}^d$ is an open subset of an abelian variety
such that the vector fields $X_{f_1},\, \cdots, \, X_{f_d}$ are linear on it.
\end{itemize}

When the number of the Poisson commuting functions satisfying the above
three conditions is less than half the 
dimension of $M$ we call it a complex partially integrable system.

\subsection{The Hitchin fibration}
We assume that $\text{genus}(X)+n\,=\, \text{genus}(X)+\sum_{i=1}^s n_i\, \geq\, 2$.

The Hitchin fibration for the moduli space of framed Higgs bundles is defined to be
the map 
\begin{equation}\label{eq Hitchin map}
h\,:\,\mathcal{M}_{H}(r,d)\,\longrightarrow\, \mathcal{H}\,:=\, \bigoplus_{i=1}^{r}H^{0}(X,
\,K^{i}_X(iD))
\end{equation}
given by the characteristic polynomial of the Higgs field, that is 
$$
h(E,\delta,\theta)(x)\,=\, \det(k 1_{E|_{x}}-\theta|_{x})=k^{r}+a_1(\theta)(x)k^{r-1}+\cdots +a_{r}(\theta)(x)
$$ 
where $k\,\in\, K_X(D)|_{x}$ and $a_{1},\,\ldots ,\, a_{r}$ are a conjugation invariant basis of
homogeneous polynomial functions on $\mathfrak{gl}(r)$. Equivalently, they are elementary symmetric polynomials on the Cartan subalgebra $\mathbb{C}^r$.

Since $H^{1}(X,\, K^{i}_X(iD))\,=\, 0$,
 for all nonnegative $i$ (recall that $\text{genus}(X)+n\, \geq\, 2$) the dimension of the vector space $\mathcal{H}$ is 
\begin{equation}\label{dimH}
N\,:=\,\dim \mathcal{H}\,=\,r^{2}(g_X-1)+\frac{r(r+1)}{2}n\, ,
\end{equation}
by the Riemann-Roch theorem. Hence $h\,=\,(h_1,\,\cdots,\,h_N)$, where $h_i$ is a
polynomial function of degree $i$.

\begin{theorem} \label{Poisson-commute} 
The above functions $h_{1}, \cdots,\, h_{N}$ Poisson commute. 
\end{theorem}
\begin{proof}
Notice that the frame does not play any role in local arguments, therefore the proof follows analogously as in \cite[Proposition 4.7.1]{Bo}. 
\end{proof}

As a corollary we recover the following result due to Bottacin \cite[Proposition 4.7.1]{Bo}.

\begin{corollary}\label{corB}
Let 
\begin{equation}\label{eq hitchin map N}
\widetilde{h}\,=\,(\widetilde{h}_1\,,\cdots,\,\widetilde{h}_N)\,:\,\mathcal{N}_H(r,d)
\,\longrightarrow\, \mathcal{H}
\end{equation}
be the Hitchin map, defined in \cite{H}, \cite{Bo}, \cite{Nitin}, which sends any pair $(E,\,\theta)$ to the
characteristic polynomial of $\theta$. Then the functions $\widetilde{h}_i$ Poisson commute.
\end{corollary}

\begin{proof}
This follows from Theorem \ref{Poisson-commute} and Theorem \ref{thm poisson map} 
after observing that
$$
\widetilde{h}(Q(E,\delta,\theta))\,=\,h(E,\delta,\theta)\, .
$$
for the forgetful map $Q$ in \eqref{eq Q}.
\end{proof}

\begin{remark}
Conversely, Corollary \ref{corB} and
Theorem \ref{thm poisson map} together imply Theorem \ref{Poisson-commute}.
\end{remark}

We next study the fibers of the map $h$. Henceforth, we will assume the divisor
$D$ to be reduced.

Let $|K_X(D)|$ denote the total space of the line bundle $K_X(D)$. This surface admits a
natural morphism
$$
\pi:|K_X(D)|\longrightarrow X.
$$
First, for each $b=(b_1,\dots,b_r)\in \mathcal{H}$, we define a divisor $X_b\subset |K_X(D)|$, called the spectral curve, as the vanishing locus of the section
$$
\lambda^r+\pi^*b_1\lambda^{r-1}+\dots+\pi^*b_r\,\in\, H^0(|K_X(D)|,\,\pi^*K_X(D)^{\otimes r})\, ,
$$
 where $\lambda\in H^0(|K_X(D)|,\,\pi^*K_X(D))$ is the tautological section.

\begin{proposition}\label{prop hitchin fibers}
The generic fiber of $h$, the Hitchin map for $\mathcal{M}_{H}(r,d)$, is a
$\mathrm{GL}(r,\mathbb{C})^{n}/\mathbb{C}^\times$ torsor over the Jacobian of the spectral curve $X_b$. More precisely, the torsor whose fiber over $L\in\mathrm{Jac}(X_b)$ is 
\begin{equation}
	{\rm Aut}((\pi_*L)_D)/\mathbb{C}^\times.
\end{equation}
\end{proposition}

\begin{proof}
Let $\widetilde{h}:\mathcal{N}_H(r,d)\longrightarrow\mathcal{H}$ be the Hitchin map for Higgs pairs. The forgetful map $Q$ defined in \eqref{eq Q} satisfies that 
$$
\widetilde{h}(Q(E,\delta,\theta))\,=\,h(E,\delta,\theta)\, .
$$
Thus $Q$ takes fibers to fibers and, generically, $\widetilde{h}^{-1}(b)$ is the
Jacobian of the spectral curve $X_b$ \cite[Proposition 3.6]{BNR}. The
proposition now follows from Proposition \ref{prop Q is torsor}.
\end{proof}

\subsection{An integrable subsystem}

There are two main problems that prevent $h$, the Hitchin map in \eqref{eq Hitchin map},  from being an algebraic 
completely integrable system:
\begin{enumerate}
\item the fibers are not abelian groups hence they are not abelian 
varieties (Proposition \ref{prop Q is torsor}); 
\item the dimension of fibers is too large as seen from Proposition \ref{dim} and the formula in (\ref{dimH}).
\end{enumerate}

In this section we define a smaller subsystem
$$
h:\mathcal{M}^\Delta_H(r,d)\longrightarrow\mathcal{H}
$$ 
whose fibers are semiabelian varieties, and whose dimension is twice the dimension of $\mathcal{H}$. The price for it is the loss of the symplectic structure, which becomes degenerate. 

In order to do this, consider the family of spectral curves parametrized by $\mathcal{H}$:
\begin{equation}
\mathcal{X}\longrightarrow\mathcal{H}.
\end{equation}
This is the subscheme of $|K_X(D)|\times\mathcal{H}$ consisting of points $(k,b)$ satisfying
$$
\lambda^r(k)+\pi^*b_1\lambda^{r-1}(k)+\dots+\pi^*b_r(k)=0.
$$
Using \cite[Proposition 3.6]{BNR}, it follows that the generic Hitchin fiber $\tilde{h}^{-1}(b)$ is isomorphic to $\mathrm{Jac}(X_b)$ the Jacobian of the spectral curve $X_{b}$. This isomorphism is called the spectral correspondence.

We can restate the above in terms of relative Jacobians.
 Consider the relative Jacobian $\mathrm{Jac}_{\mathcal{H}}(\mathcal{X})$, by which we mean stable relative line bundles. Stability is defined so that $\mathrm{Jac}_{\mathcal{H}}(\mathcal{X})$ parametrizes, via the spectral correspondence, stable Hitchin pairs. The spectral correspondence gives a morphism
\begin{equation}\label{eq spectral corr}
s\, :\, \mathrm{Jac}_{\mathcal{H}}(\mathcal{X})\longrightarrow \mathcal{N}_H(r,d).
\end{equation}

Let $\widetilde{D}\,:=\,\pi^*D\,\subset\,|K_X(D)|$ and consider relative framed line bundles on
 $\mathcal{X}$, that is, pairs $(L,\delta')$ for $L\in\mathrm{Jac}_{\mathcal{H}}(\mathcal{X})$, such that 
$$
L_b|_{\widetilde{D}_b}\cong \mathcal{O}_{\widetilde{D}_b}\, ,
$$
where $X_b:=\mathcal{X}|_{\{b\}}$, $L_b=L|_{X_b}$ and $\widetilde{D}_b\,=\,
\widetilde{D}\cap X_b$.
We will say that such a framed line bundle $(L,\delta')$ is stable if the underlying line
bundle is so, or in other words, if $L$ yields a stable framed Higgs bundle under the spectral correspondence.

\begin{proposition}\label{prop framed rel line bundles}
Let $\mathcal{H}_{nr}\subset {\mathcal{H}}$ denote the subset of the Hitchin base corresponding to smooth spectral curves which are unramified over $D$. Then, relative framed line bundles on $\mathcal{X}|_{\mathcal{H}_{nr}}$ are parametrized by a $(\mathbb{C}^\times)^{nr-1}$ torsor $\mathcal{P}$ over $\mathrm{Jac}_{\mathcal{H}}(\mathcal{X}|_{\mathcal{H}_{nr}})$.
\end{proposition}
\begin{proof}The proof that this defines a torsor is very similar to that of Proposition \ref{prop Q is torsor} and is thus omitted. 
\end{proof}

We also get:

\begin{lemma}
The spectral correspondence induces a morphism
\begin{equation}\label{eq spectral corr M}
\sigma:\mathcal{P}\to\mathcal{M}_H(r,d).
\end{equation}
\end{lemma}

\begin{definition}
The image $\sigma(\mathcal{P})\,\subset\, \mathcal{M}_H(r,d)$, under the spectral
correspondence \eqref{eq spectral corr M}, will be called the moduli
of {\it diagonally framed Higgs pairs}, and it will be denoted by
$\mathcal{M}^\Delta_H(r,d)$.
\end{definition}

Next we study the subvariety $\mathcal{M}^\Delta_H(r,d)$ infinitesimally. Given
$(E,\delta,\theta)\in\mathcal{M}_H^\Delta(r,d)$, we have that over each $x_i\in D$, there are
distinguished one dimensional subspaces $L_{j}^i$ of $E_{x_{i}}$ such that $\delta=\pi_*\delta'$ is diagonal for the identification $E_{x_i}\cong\oplus_jL_{j}^i$. Let $\pi_j:E_{x_i}\twoheadrightarrow L_{j}^i$ be the projection. Let $Diag_E\subset\mathrm{End}(E)|_D$ be the subset of endomorphisms which are diagonal in the preferred reference, namely,  $\{L_j^i \, :\, j=1,\dots, r\}$ over $x_i$. Let 
$$
\pi: \mathrm{End}(E)\twoheadrightarrow Diag_E
$$
be the morphism associating to $f$ the map
$$
\pi(f)\, :\, v\mapsto(\pi_1(f(\pi_1(v))),\dots,\pi_r(f(\pi_r(v))) )
$$
where $v\in E_{x_i}$.

Consider the complex
$$
\mathcal{C}_\bullet^\Delta:\mathrm{Ker(\pi)}\stackrel{[\theta,\cdot]}{\longrightarrow}\mathrm{End}(E)\otimes K_X(D).
$$
\begin{lemma}
Infinitesimal deformations of 
$(E,\delta,\theta)$ along $\mathcal{M}_H^\Delta(r,d)$ are parametrized by $\mathbb{H}^1(\mathcal{C}_\bullet^\Delta)$.

In particular, the dimension of $\mathcal{M}_H^\Delta(r,d)$ is 
$$\dim \mathcal{M}^\Delta_H(r,d)=2\dim\mathcal{H}.
$$
\end{lemma}
\begin{proof}
The proof that the complex parametrizes the right deformations follows the same arguments as Lemma \ref{lem2}.

We see that $\mathbb{H}^0(\mathcal{C}_\bullet^\Delta)=0$ just as in Corollary \ref{cor0}. If the underlying bundle is stable, we have moreover that 
$\mathbb{H}^2(\mathcal{C}_\bullet^\Delta)=0$ (as the dual complex has no zero hypercohomology). 

So the tangent space at a point with underlying stable bundle is identified with $\mathbb{H}^1(\mathcal{C}^\Delta_\bullet)$.

To compute dimensions, by vanishing of all hypercohomology groups other than the first one, it follows that
$$
\dim\mathbb{H}^1(\mathcal{C}^\Delta_\bullet)=\chi(\mathrm{End}(E)\otimes K_X(D))-\chi(\mathrm{Ker}(\pi)).
$$
By definition 
$$
\chi(\mathrm{Ker}(\pi))=\chi(\mathrm{End}(E))-\chi(Diag_E)=-r^2(g-1)-nr.
$$
Also:
$$
\chi(\mathrm{End}(E)\otimes K_X(D))=r^2(g-1+n).
$$
So 
$$
\dim\mathbb{H}^1(\mathcal{C}^\Delta_\bullet)=r^2(2g-2)+nr(r+1)
$$
and the statement follows.
\end{proof}

\begin{remark}
The dimension of $\mathcal{M}_H^{\Delta}(r,d)$ can also be computed using that $\mathcal{M}^\Delta_H(r,d)$ is a $(\mathbb{C}^{\times})^{nr-1}$ torsor over $\mathcal{N}_H(r,d)$. Indeed
$$
\dim \mathcal{M}_H^{\Delta}(r,d)=\dim \mathcal{P}=nr-1+\dim \mathrm{Jac}_{\mathcal{H}}(\mathcal{X}).
$$
Also, by \cite[(3.1.8)]{Bo}
$$
\dim \mathrm{Jac}_{\mathcal{H}}=\dim\mathcal{N}_H(r,d)=r^2(2g-2+n)+1.
$$
Finally,
$$\dim \mathcal{M}_H^{\Delta}(r,d)=2\left(r^2(g-1)+n\frac{r(r+1)}{2}\right)=2\dim\mathcal{H}.
$$
\end{remark}

\begin{remark}\label{rk spectrally framed poisson}
Note that as non-degeneracy of $\Psi$ is a consequence of Serre duality 
$$
H^1(\mathrm{End}(E)(-D))\cong H^0(\mathrm{End}(E)\otimes K_X(D))^*, 
$$
and $h^1(\mathrm{Ker}(\pi))<h^1(\mathrm{End}(E)(-D))$ whenever $E$ is stable, it follows that $\Psi|_{\mathcal{M}_H^{\Delta}}$ may degenerate on a subspace.
\end{remark}

\begin{proposition}\label{prop integrable subsystem}
	Let
	 $$
	 h^\Delta:\mathcal{M}_H^{\Delta}(r,d)\longrightarrow\mathcal{H}
	 $$
	 be the restriction of the Hitchin map. Then:
\begin{enumerate}
\item{} $h^\Delta$ is surjective,

\item{} the generic fibers are semiabelian varieties, i.e. an extension of an abelian variety by a torus. More precisely 
$$
(h^\Delta)^{-1}(b)=
\{(L,\gamma)\ :\ L\in\mathrm{Jac}(X_{b}),\gamma=
(\gamma_1,\dots,\gamma_r); \gamma^{j}_i:L_{y^{j}_i}\cong \mathcal{O}_{y^{j}_i}, y^{j}_i\in
\pi_b^{-1}(x_i)\}/\mathbb{C}^\times\, .
 $$
\end{enumerate}
\end{proposition}

\begin{proof}
 Surjectivity follows immediately from that of $\widetilde{h}$. As for
semiabelian-ness of the fibers, it follows from Proposition \ref{prop framed rel line bundles},
together with \cite[Proposition 7.2.1]{BSU}. 
\end{proof}

\subsection{Comparison of two integrable systems: completing the Hitchin system}

In this section we compare the Hitchin systems $h$ and $\widetilde{h}$ defined in \eqref{eq 
Hitchin map} and \eqref{eq hitchin map N} respectively.

First of all, note that the symplectic structure on $\mathcal{M}_H(r,d)$ is compatible with the Poisson structure of $\mathcal{N}_H(r,d)$. Now, in what follows we argue that this is done somewhat perpendicularly to the Hitchin fibers of $\mathcal{N}_H(r,d)$. In particular, in order to obtain extra action angle coordinates to (locally) recover, for example, the semiabelian subvarieties appearing in Proposition \ref{prop integrable subsystem}, new functions need to be added to the system.

By Proposition \ref{prop Q is torsor}, we have that locally
$$
\mathcal{O}_{\mathcal{M_H}}=\mathcal{O}_{\mathcal{N}_{H}}\otimes_{\mathbb{C}}\mathcal{O}_{\mathrm{GL(r,\mathbb{C})/\mathbb{G}_m}}\, .
$$
 Consider
the embedding 
\begin{equation}\label{eq embedding into PGL}
\mathrm{GL}(r,\mathbb{C})^n/\mathbb{C}^\times\hookrightarrow \mathrm{PGL}(rn,\mathbb{C}).
\end{equation}

\begin{proposition}\label{prop completed int system}Let $f\in \mathcal{O}_{\mathrm{GL(r,\mathbb{C})/\mathbb{G}_m}}$ be a local function on the fibers of $Q$. Then
$$
\{f,h_i\}=0
$$
for all $i=1,\dots, N$.
In particular, consider the (local) coordinate functions 
$x_{i}$ on the maximal torus of an affine open subset of $\mathrm{GL}(r,\mathbb{C})^n/\mathbb{C}^\times$ (defined for instance via the embedding \eqref{eq embedding into PGL}). Then the following set of local functions Poisson commutes
$$
\{x_{i}, h_k\}=0
$$
for $1\leq i\leq nr-1,k=1,\dots,N$.

Moreover, the associated Hamiltonian vectors are linear on the fibers of the local map 
$$
(h_1,\dots, h_N,x_{1},\dots, {x_{nr-1}})\,:\,\mathcal{M}_H(r,d)\longrightarrow\mathbb{C}^{r^2(g-1)+\frac{r(r+1)}{2}n+nr-1}.
$$
\end{proposition}

\begin{proof}
Since the computation is local, we may assume the torsor \eqref{eq Q} is trivial. 

Consider the following commutative diagram:
\begin{equation}\label{eq tangent spaces}
\xymatrix{
&&H^0({\rm End}(E)|_D)/\mathbb{C}\ar[d]_i\\
H^0({\rm End}(E)\otimes K_X(D))\ar[r]\ar@{=}[d]&\mathbb{H}^1(\mathcal{C}_\bullet) \ar[r]\ar[d]_{dQ}&H^1({\rm End}(E)(-D))\ar[d]\\
H^0({\rm End}(E)\otimes K_X(D))\ar[r]&\mathbb{H}^1(\mathcal{C}_\bullet^\tau) \ar[r]&H^1({\rm End}(E))
}
\end{equation}
The second and third lines give a decomposition of the tangent spaces at a point $(E,\delta,\theta)\in\mathcal{M}_H(r,d)$ and $Q(E,\delta,\theta)=(E,\theta)\in\mathcal{N}_H(r,d)$ respectively, while the first row is the tangent space of $Q^{-1}(E,\theta)$. 

Take $v\in T_{(E,\delta,\theta)}\mathcal{M}_H$ and represent it as $(\mu_{ij},\eta_i)$. Note that
$$
\mu_{ij}=(\sigma_i-\sigma_j,\, \rho_{ij})\, ,
$$
where $\sigma_i\in H^0({\rm End}(E)(-D))$ are local functions with the same restriction to $D$,
namely $\sigma_i-\sigma_j\in \mathrm{Im}\left\{H^0({\rm End}(E)|_D)\longrightarrow H^1({\rm End}(E)(-D))\right\}$,
and $\rho_{ij}\in H^1({\rm End}(E))$.

Now, clearly, for $f\in \mathcal{O}_{\mathrm{GL(r,\mathbb{C})/\mathbb{G}_m}}$, we have
$$
df(\mu_{ij},\eta_i)=df(\sigma_i-\sigma_j,\eta_i)=(df(\sigma_k|_D),0).
$$
In order to assign to $df$ an element of
 $\mathbb{H}^1(\mathcal{C}^{\bullet})$, consider the diagram dual to \eqref{eq tangent spaces}:
\begin{equation}\label{dual commutative diag}
\xymatrix{
V&&\\
H^0({\rm End}(E)\otimes K_X(D))\ar[r]\ar[u]&\mathbb{H}^1(\mathcal{C}^\bullet) \ar[r]&H^1({\rm End}(E)(-D))\\
H^0({\rm End}(E)\otimes K_X)\ar[u]\ar[r]&\mathbb{H}^1(\mathcal{C}_\bullet^\tau)^* \ar[r]\ar[u]_{dQ^t}&H^1({\rm End}(E)(-D))\ar@{=}[u]
}
\end{equation}
where $V\subset H^0({\rm End}(E)\otimes K_X(D)|_D)$ is the image of the restriction map $$H^0({\rm End}(E)\otimes K_X(D))\to
H^0({\rm End}(E)\otimes K_X(D)|_D)\, .$$ Note that
Serre duality pairing between $H^0(\mathrm{End}(E)\otimes K_X(D))$ and
$H^1(\mathrm{End}(E)(-D))$ identifies
$$H^0(\mathrm{End}(E)|_D)\, \stackrel{\sim}{\longrightarrow}\,
H^0(\mathrm{End}(E)\otimes K_X(D)|_D)^*,\,
H^0(\mathrm{End}(E)\otimes K) )\, \stackrel{\sim}{\longrightarrow}\,
H^1(\mathrm{End}(E))^*\, .$$ 

So $df$ is represented by an element
$(\xi_i,0)$ where $\xi_i\in V\subset H^0(\mathrm{End}(E)\otimes K_X(D)|_D)$.
We have $(\xi_i,0)\in\mathbb{H}^1(\mathcal{C}^\bullet)$, as $\xi_i-\xi_j
\,=\,[0,\,\theta]$.

Now, $dh_k$ is represented by $(0,k\theta^{k-1})$, where $k\theta^{k-1}$ is interpreted as a 1-cocycle of $\mathrm{End}(E)(-D)$ as follows:  let $\{ U_{i}\}_{i=1}^{d}$ be a open covering of $X$, choose a point $p\in X$ with local coordinate $z$ on $U_{ij}$, the intersection of $U_{i}$ and $U_{j}$, such that $K_X(D)|_{U_{ij}}\cong\mathcal{O}_{U_{ij}}$. Let $U'_{ij}=U_{ij}\setminus\{p\}$. Then
$$
\theta^{k-1}|_{U'_{ij}}\in H^0(U'_{ij},\mathrm{End}(E)).
$$
Moreover, $K|_{U_{ij}}\cong\mathcal{O}_{U_{ij}}(-D)$, so the meromorphic form $dz/z$ generates
$H^1(X,\,K_X)$ and moreover $\theta^kdz/z\in \Gamma(U_{ij},\mathrm{End}(E)(-D))$. We may check
that this defines a $1$-cocycle. 

Now, on $U_{ij}$, we have $K\cong\mathcal{O}(-D)$, and the pairing of $\theta^kdz/z$ with any element of $V$ will be identically zero. It thus follows that 
$$
\{f,h_k\}=0.
$$

Finally, since Serre duality restricts on the fibers of  $Q$  to the 
Killing form, the functions on the moduli space, given by the coordinate functions of
the maximal torus, Poisson commute.

The last statement follows because of (a) linearity of $X_{h_k}$ on the Jacobian of the 
spectral curve, (b) vanishing of $X_{h_k}$ on the fibers of $Q$ , (c) linearity of 
$X_{x_i}$ on $(\mathbb{C}^\times)^{nr-1}$ and (d) vanishing of $X_{x_i}$ on 
$\mathrm{Jac}(X_b)$.
\end{proof}

\section*{Acknowledgements}

The first author is partially supported by a J. C. Bose Fellowship. The second author wishes to thank Tata Institute for Fundamental Research for its hospitality while part of this work was 
developed. The third author was supported by the Marie Curie Project GEOMODULI of the Programme FP7/PEOPLE/2013/CIG, Project 
Number 618471 and the Fondation Nationale des Sciences, National Centre of Competence in Research, SwissMAP project. 


\end{document}